\documentclass[11pt]{amsart}
\usepackage[margin=2.5cm]{geometry}

\usepackage{amsmath,amsthm,amsfonts,latexsym,amscd,amssymb,enumerate,color,hyperref,mathtools,bbm}
\input{xypic}

\usepackage{lineno}

\usepackage{rotating}
\usepackage{multirow}
\usepackage{graphicx}
\usepackage{epstopdf}
\usepackage[usenames,dvipsnames]{xcolor}
\usepackage{tikz}
\usepackage{inputenc}
\usepackage{tikz-cd}

\usepackage{booktabs}
\usepackage{enumitem}
\usepackage{multicol}

\usepackage{etoolbox}
\newcounter{dummy}
\makeatletter  
\newcommand\myitem[1][]{\item[#1]\refstepcounter{dummy}\def\@currentlabel{#1}}   
\makeatother

\usepackage{hyperref}
\hypersetup{colorlinks,allcolors=blue}

\def \<{\langle}
\def \>{\rangle}

\newcommand{\Z}{\mathbb{Z}}
\newcommand{\R}{\mathbb{R}}
\newcommand{\NN}{\mathbb{N}}

\newcommand{\DD}{\mathcal{D}}
\newcommand{\Di}[1]{\DD_\infty(#1)}

\newcommand{\vn}{\varnothing}

\newcommand{\mset}[1]{\left\{\!\left\{ #1 \right\}\!\right\}}

\newtheorem{thm}{Theorem}[section]

\newtheorem{lem}[thm]{Lemma}
\newtheorem{prop}[thm]{Proposition}

\newtheorem{theorem}{Theorem}

\newtheorem{corollary}[theorem]{Corollary}

\theoremstyle{definition}
\newtheorem{defn}[thm]{Definition}

\newtheorem{example}[thm]{Example}

\newtheorem*{ack}{Acknowledgements}

\numberwithin{equation}{section}
\numberwithin{table}{section}

\makeatletter
\def\@setthanks{\vspace{-\baselineskip}\def\thanks##1{\@par##1}\thankses}
\makeatother

\begin{document}

\author[M.~Che]{Mauricio Che $^\mathrm{A}$}
\address[Che]{Department of Mathematical Sciences, Durham University, United Kingdom.}
\email{mauricio.a.che-moguel@durham.ac.uk}

\thanks{$^{\mathrm{A}}$Supported by CONACYT Doctoral Scholarship No. 769708.}

\author[F.~Galaz-Garc\'ia]{Fernando Galaz-Garc\'ia $^{\mathrm{B}}$}
\address[Galaz-Garc\'ia]{Department of Mathematical Sciences, Durham University, United Kingdom.}
\email{fernando.galaz-garcia@durham.ac.uk}

\author[L.~Guijarro]{Luis Guijarro $^{\mathrm{B}}$}
\address[Guijarro]{Department of Mathematics, Universidad Aut\'onoma de Madrid and ICMAT CSIC-UAM-UC3M, Spain}
\email{luis.guijarro@uam.es} 

\thanks{$^{\mathrm{B}}$Supported by research grants  MTM2017-‐85934-‐C3-‐2-‐P and PID2021-124195NB-C32 from the  Ministerio de Econom\'{\i}a y Competitividad de Espa\~{na} (MINECO), and by ICMAT Severo Ochoa project CEX2019-000904-S(MINECO)}

\author[I.~Membrillo Solis]{Ingrid Membrillo Solis $^\mathrm{C}$}
\address[Membrillo Solis]{Mathematical Sciences, University of Southampton, United Kingdom}
\email{i.membrillo-solis@soton.ac.uk} 
\curraddr{School of Computer Science and Engineering, University of Westminster, United Kingdom}
\email{i.membrillosolis@westminster.ac.uk}

\thanks{$^\mathrm{C}$Supported by the Leverhulme Trust (grant RPG-2019-055).}

\author[M.~Valiunas]{Motiejus Valiunas}
\address[Valiunas]{Mathematical Institute, University of Wroc{\l}aw, Poland}
\email{motiejus.valiunas@math.uni.wroc.pl} 

\title[Geometry of the Bottleneck Distance]{Basic Metric Geometry of the Bottleneck Distance}
\date{\today}

\begin{abstract}
Given a metric pair $(X,A)$, i.e.\ a metric space $X$ and a distinguished closed set $A\subset X$, one may construct in a functorial way a pointed pseudometric space $\DD_\infty(X,A)$ of persistence diagrams equipped with the bottleneck distance.
We investigate the basic metric properties of the spaces $\DD_\infty(X,A)$ and obtain characterizations of their metrizability, completeness, separability, and geodesicity. 
\end{abstract}

\subjclass[2010]{53C23, 55N31, 54F45}

\keywords{bottleneck distance, metric pair, persistence diagram}
\maketitle

\setcounter{tocdepth}{1}


\section{Introduction}
\label{S:INTRO}
In recent years, persistent homology has received considerable attention as it has shown not only to be a powerful tool in data analysis but to provide new perspectives in the study of abstract metric spaces. In persistent homology, the small, medium and large-scale topological features of geometric objects are encoded in a \textit{persistence diagram}, a multiset of points in the (extended) plane. 
The construction of such diagrams starts with a filtration of a topological space, and records the exact birth and death times of  generators of the persistent homology associated to that particular filtration \cite{ZC05}. Most common filtrations come from the Vietoris--Rips complexes over a given metric space or from sublevel sets of natural functions defined on some topological space. In either case, the obtained persistence diagram contains topological information regarding different levels or scales of the underlying problem of interest \cite{CVJ22}. In applications to data analysis, it is crucial to compare the topological features of the analyzed data quantitatively, which in turn might be filtrated or sampled using different methods. Therefore, most applications of persistent homology require not only being able to compute persistence diagrams, but a measure of similarity between such objects.

The bottleneck distance was the first measure of similarity that appeared in the literature to compare persistence diagrams \cite{C-SEH07}. It is well-known that the bottleneck distance is an (extended) pseudometric \cite{oudot2017}. The successful introduction of the bottleneck distance gave rise to  alternative similarity measures for persistence diagrams. For instance, the $p$-Wasserstein metrics, for $1\leq p < \infty$, form  
a one-parameter family of similarity measures between persistence diagrams, which, in contrast to the bottleneck distance, are metrics   
\cite{CEHM10}; in the limit, one may think of the bottleneck distance as another member of this family.
Out of all the similarity measures that have been defined in persistent homology, only the bottleneck distance has been shown to be equivalent to the interleaving distance, a pseudometric defined for persistent modules, the algebraic analogues of  persistence diagrams \cite{oudot2017}.

Several authors have studied the topological and geometric properties of the spaces defined by the bottleneck distance,  $\DD_\infty(\mathbb R^2_{\geq0},\Delta)$, and the $p$-Wasserstein metrics, $\DD_p(\mathbb R^2_{\geq0},\Delta)$, where $\R^{2}_{\geq 0} = \{(x,y)\in\R^{2} : 0\leq x\leq y\}$, $\Delta= \{(x,y)\in\R^{2}_{\geq 0} : x=y\}$, and $1 \leq p <\infty$,  as well as those of generalizations of these spaces \cite{bubenik2,bubenik1,CGGGMS,divol21}. In \cite{BH,CGGGMS}, the authors study the geometry and topology of spaces of generalized persistence diagrams $\DD_p{(X,A)}$, $1\leq p<\infty$. These spaces are a generalization of the spaces defined in \cite[Definition 4]{mileyko1}. When   $(X,A)=(\mathbb \R^2_{\geq0},\Delta)$, one recovers the usual space of persistence diagrams, where the coordinates of a  point $(x,y)\in \R^2_{\geq0}$ are interpreted as the birth at scale $x$ and death at scale $y$ of a topological feature of the space under consideration. The  spaces of  persistence diagrams are constructed in a functorial way: for each $p\in [1,\infty]$, there exists a functor $\DD_p$ which assigns to a metric pair $(X,A)$, i.e.\ a metric space $X$ and a distinguished non-empty closed subset $A\subset X$, a space of (generalized) persistence diagrams $\DD_p(X,A)$ equipped with either the bottleneck distance, when $p=\infty$, or the $p$-Wasserstein metric, when $1\leq p<\infty$. The  topology and geometry of the spaces $\DD_p(\mathbb R^2_{\geq0},\Delta)$, $1\leq p<\infty$,  were analyzed in \cite{mileyko1,turner14}. We note, however, that in previous works, the geometric and topological properties of spaces of persistence diagrams equipped with the bottleneck distance had not been completely investigated, since their study requires the use of distinct methods from those used in the study of the $p$-Wasserstein spaces of persistence diagrams. 
A further motivation for the study of the spaces $\DD_p(X,A)$ is that, when $X=\R^{2n}_{\geq 0} = \{(x_1,y_1,\dots,x_n,y_n):0\leq x_i\leq y_i\ \text{for}\ i=1,\dots,n\}$ and $A=\Delta_n = \{(x_1,y_1,\dots,x_n,y_n)\in \R^{2n}_{\geq 0}: x_i=y_i\ \text{for}\ i=1,\dots, n\}$ for $n\geq 2$, one may consider the resulting spaces of generalized persistence diagrams as the parameter spaces for rectangle-decomposable persistent modules. These objects arise in multiparameter persistent homology and have received considerable attention in recent years (see, for example, \cite{asa2022,bak2021,botnan2022,SC17}).

The present work aims to fill some gaps in our basic understanding of the topological and geometric properties of the spaces of persistence diagrams equipped with the bottleneck distance. First of all, we note that, given a metric pair $(X,A)$, the space $\Di{X,A}$ is not necessarily a metric space: the (generalized) bottleneck distance function $d_\infty$ might be a pseudometric that is not a metric. This appears to be the case for most common examples of spaces of (generalized) persistence diagrams equipped with the bottleneck distance, including the usual space of euclidean persistence diagrams. 


\begin{theorem}[Metrizability] \label{thm:Di-metric}
The pseudometric space $\Di{X,A}$ is a metric space if and only if $X \setminus A$ with the restricted metric of $X$ is a discrete space.
\end{theorem}

 If $X$ is complete, then $\DD_p(X,A)$ is a complete metric space provided $1\leq p<\infty$ (see \cite[Theorem B (1)]{CGGGMS}). Here, we obtain the following characterization of completeness for $\DD_p(X,A)$ for all $p\in [1,\infty]$.  


\begin{theorem}[Completeness] \label{thm:Di-complete}
For any $p\in [1,\infty]$, the space $\DD_p(X,A)$ is complete if and only if $X/A$ is complete.
\end{theorem}


\begin{corollary} \label{cor:Di-complete}
For any $p\in [1,\infty]$, if $X$ is complete, then so is $\DD_p(X,A)$.
\end{corollary}

We also consider the separability of $\DD_\infty(X,A)$.  For  $1\leq p < \infty$, if $(X,A)$ is separable, then $\DD_p(X,A)$ is also separable (see \cite[Theorem B (2)]{CGGGMS}). This is no longer the case for $\DD_\infty(X,A)$. Indeed, for instance, the usual space of euclidean persistence diagrams with the bottleneck distance, $\DD_\infty(\mathbb R^2_{\geq 0},\Delta)$, is not separable (see Example~\ref{ex:usual-not-separable}), although this particular fact seems to be  known already (see \cite[Theorem~5]{BV18}). Here, we obtain the following characterization of separability for $\DD_\infty(X,A)$. Given $r > 0$ and a subset $Y \subseteq X$, we  let $B_r(Y)$  denote the open $r$-neighborhood of $Y$ in $X$, i.e.\ $B_r(Y)=\{ x \in X : d(x,y) < r \text{ for some } y \in Y \}$. Given $y \in X$, we write $B_r(y)$ for $B_r(\{y\})$.


\begin{theorem}[Separability] \label{thm:Di-separable}
The pseudometric space $\Di{X,A}$ is separable if and only if the subset $B_D(A) \setminus B_\delta(A) \subseteq X$ is totally bounded for all $D > \delta > 0$.
\end{theorem}

Finally, we give a criterion for $\Di{X,A}$ to be a geodesic pseudometric space when $X$ is proper. Given $x,y \in X$, we call a rectifiable path $\gamma\colon [0,1] \to X$ a \emph{geodesic} from $x$ to $y$ if $\gamma(0) = x$, $\gamma(1) = y$, and $\gamma$ has length $d(x,y)$. A pseudometric space $X$ is \emph{geodesic} if any two points of $X$ can be joined by a geodesic.


\begin{theorem}[Geodesicity] \label{thm:Di-geodesic}
Let $X$ be a proper metric space. Then the pseudometric space $\Di{X,A}$ is geodesic if and only if any two points $x,y \in X/A$ such that either $d(x,y) < d(x,A)$ or $y=A$ can be joined by a geodesic in $X/A$.
\end{theorem}


\begin{corollary} \label{cor:Di-geodesic}
If $X$ is proper and geodesic, then $\Di{X,A}$ is geodesic.
\end{corollary}

Our results are related to some of the theorems in \cite{BH}, where the authors study the space $\overline{D}_\infty(X,A)$ of persistence diagrams $\sigma$ such that the subdiagram $\mset{x \in \sigma : d(x,A) \geq \delta}$ is finite for every $\delta > 0$. We point out that $\overline{D}_\infty(X,A)$ is a closed subspace of $\Di{X,A}$.  However, the spaces $\Di{X,A}$ and $\overline{D}_\infty(X,A)$ do not necessarily satisfy the same metric properties. For instance, if $X$ is a metric space then so is $\overline{D}_\infty(X,A)$ (see the proof of \cite[Lemma~2.4(2)]{CGGGMS}); the analogous statement for $\Di{X,A}$, however, is far from being true (see Theorem~\ref{thm:Di-metric}).
In \cite[Theorem~6.2]{BH}, the authors showed that if $X$ is separable then so is $\overline{D}_\infty(X,A)$, but this is not necessarily true for $\Di{X,A}$ (see Example~\ref{ex:usual-not-separable}).
Moreover, if $A$ is distance minimizing (that is, for every $x \in X$ there exists $a \in A$ such that $d(x,a) = d(x,A)$) then there always exists an optimal bijection between two diagrams in $\overline{D}_\infty(X,A)$ \cite[Theorem~4.9]{BH}, and if in addition $X$ is geodesic then so is $\overline{D}_\infty(X,A)$ \cite[Theorem~5.10]{BH}. These two statements are no longer true for $\Di{X,A}$: see Example~\ref{ex:non-geodesic}.
The characterization of completeness of $\overline{D}_\infty(X,A)$ (see \cite[Theorem~6.3 \& Proposition~6.10]{BH}) does agree with the characterization for $\Di{X,A}$ (see Theorem~\ref{thm:Di-complete}); however, the proof given in \cite{BH} does not seem to be easy to adapt for the space $\Di{X,A}$.

It is also worth mentioning that in \cite{BH} the authors consider extended pseudometric spaces (see Definition~\ref{d:metric}), which are not necessarily metric spaces. Our results also hold if we replace metric pairs with pseudometric ones, and our proofs can be readily adapted with minimal modifications. The sole exception is the proof of Theorem~\ref{thm:Di-geodesic} (see Section~\ref{s:geodesicity}), which relies on the uniqueness of ultralimits in metric spaces. Ultralimits in a pseudometric space need not be unique. Nevertheless, our arguments remain valid if we select ultralimit points arbitrarily from their possible values. For simplicity, we have stated all our results for metric pairs.

Our article is organized as follows. In Section~\ref{s:preliminaries}, we recall the definition of the functors $\DD_p$, along with necessary background on persistence diagram spaces. Section~\ref{s:metrizability} contains the proof of Theorem~\ref{thm:Di-metric}. We then prove Theorem~\ref{thm:Di-complete} and Corollary~\ref{cor:Di-complete} in Section~\ref{s:completeness}. Finally, we prove Theorem~\ref{thm:Di-separable} in Section~\ref{s:separability}, and Theorem~\ref{thm:Di-geodesic} together with Corollary~\ref{cor:Di-geodesic} in Section~\ref{s:geodesicity}.

\begin{ack}
The authors wish to thank the referees for their helpful comments.
\end{ack}


\section{Preliminaries}
\label{s:preliminaries}

In this section, we collect preliminary material that we will use in the rest of the article. The contents of this section are based on \cite{CGGGMS}, where the reader may find further details. We refer the reader to \cite{BBI} for basic results on the geometry of metric spaces.


\begin{defn} \label{d:metric}
Let $X$ be a set.
\begin{enumerate}

\item A map $d\colon X\times X\to[0,\infty]$ is an \emph{extended pseudometric} on $X$ if it is symmetric and satisfies the triangle inequality, and a \emph{pseudometric} on $X$ if, in addition, $d(x,y)< \infty$ for $x,y\in X$.

\item An (extended) pseudometric $d$ on $X$ such that $d(x,y) = 0$ if and only if $x=y$ is an (\emph{extended}) \emph{metric}. When $d$ is an (extended) pseudometric on $X$, points at distance zero from each other give a partition of $X$, and $d$ induces an (extended) metric in the corresponding quotient set.

\item An (\emph{extended}) \emph{pseudometric space} $(X,d_X)$ is a set $X$ together with an (extended) pseudometric $d_X$. An (\emph{extended}) \emph{metric space} $(X,d_X)$ is a set $X$ together with an (extended) metric $d_X$.
\end{enumerate}
\end{defn}


\begin{defn}\label{d:metric pairs}
A \emph{metric pair} $(X,A)$ consists of a metric space $(X,d_X)$ and a non-empty closed subset $A\subseteq X$. 
 When $A$ is a point, we will say that $(X,\{a_0\})$ is a \emph{pointed metric space with basepoint} $a_0$. 
\end{defn}

We now define the space $\DD_\infty(X,A)$ of $\infty$-persistence diagrams on a metric pair $(X,A)$.
We will denote multisets by using two curly brackets $\mset{\cdot}$. Persistence diagrams will be denoted by Greek letters.

Let $(X,d)$ be a metric space. We denote by $\widetilde{\DD}(X)$ the set of countable multisets $\mset{x_1,x_2,\dots}$ of elements of $X$, and equip it with the  \emph{bottleneck distance} $\widetilde d_\infty$ given by 
\begin{equation}
\widetilde d_\infty(\widetilde\sigma,\widetilde\tau)=\inf_{\phi\colon \widetilde\sigma \to \widetilde\tau} \sup_{x \in \widetilde\sigma} d(x,\phi(x)),
\end{equation}
where $\phi$ ranges over all bijections between $\widetilde\sigma$ and $\widetilde\tau$ in $\widetilde{\DD}(X)$. Here, by convention, we set $\inf \varnothing = \infty$, that is, $\widetilde d_\infty(\widetilde\sigma,\widetilde\tau) = \infty$ whenever $\widetilde\sigma$ and $\widetilde\tau$ do not have the same cardinality. Observe that $\widetilde{d}_\infty$ is an extended pseudometric on $\widetilde{\DD}(X)$, making $(\widetilde\DD_\infty(X),\widetilde{d}_\infty)$ into an extended pseudometric space. We denote the empty multiset by $\widetilde \sigma_\varnothing$.

We may similarly define the \emph{$p$-Wasserstein distance} $\widetilde d_p$ by
\[
\widetilde d_p(\widetilde\sigma,\widetilde\tau) = \inf_{\phi\colon \widetilde\sigma\to \widetilde\tau} \left(\sum_{x\in \widetilde\sigma} d(x,\phi(x))^p\right)^{1/p}
\]
for any $p\in [1,\infty)$, which are extended pseudometrics on $\widetilde\DD(X)$ as well.

Given $\widetilde\sigma,\widetilde\tau \in \widetilde{\DD}(X)$, we define an equivalence relation on $\widetilde{\DD}(X)$ by setting $\widetilde\sigma \sim_A \widetilde\tau$ whenever there exist $\widetilde\alpha,\widetilde\beta \in \widetilde{\DD}(A)$ such that $\widetilde\sigma \sqcup \widetilde\alpha = \widetilde\tau \sqcup \widetilde\beta$. 
Let $\DD(X,A)$ be the quotient set $\widetilde{\DD}(X)/{\sim_A}$. Given $\widetilde\sigma \in \widetilde{\DD}(X)$, we let $\sigma$ be the equivalence class of $\widetilde\sigma$ in $\DD(X,A)$. Observe that $\sigma_\vn$, the equivalence class of $\widetilde\sigma_\vn$, is also the equivalence class of any $\widetilde\alpha\in \widetilde\DD(A)$.

The maps $\widetilde d_p$ induce functions $d_p\colon \DD(X,A) \times \DD(X,A) \to [0,\infty]$ given by
\begin{equation}
d_p(\sigma,\tau)= \inf_{\widetilde\alpha,\widetilde\beta\in\widetilde{\DD}(A)}\widetilde d_p(\widetilde\sigma \sqcup \widetilde\alpha,\widetilde\tau \sqcup \widetilde\beta).
\end{equation}


\begin{defn}\label{def:spaceofpd}
For any $p\in [1,\infty]$, the \emph{space of $p$-persistence diagrams} $\DD_p(X,A)$ on the pair $(X,A)$ is the set of all $\sigma \in \DD(X,A)$ such that $d_p(\sigma,\sigma_\varnothing) < \infty$.
\end{defn}

In particular, we have $d_p(\sigma,\sigma_\varnothing) = \left( \sum_{x \in \widetilde{\sigma}} d(x,A)^p \right)^{1/p}$ for any $p \in [1,\infty)$ and $d_\infty(\sigma,\sigma_\varnothing) = \sup_{x \in \widetilde{\sigma}} d(x,A)$. The inequality $d_p(\sigma,\sigma_\varnothing) < \infty$ may be interpreted as the $p$-norm of the distances of points in $\sigma$ to the set $A$ (the total $p$-persistence, in the original setting of persistence homology) being finite. 


\begin{prop}\label{lem:pair}
For any $p\in [1,\infty)$, the function $d_p$ is an extended metric on $\DD(X,A)$ and a metric on $\DD_p(X,A)$. On the other hand, the function $d_\infty$ is an extended pseudometric on $\DD(X,A)$ and a pseudometric on $\DD_\infty(X,A)$.
\end{prop}

For simplicity, we will treat elements in $\DD(X,A)$ as multisets by themselves, with the understanding that, whenever we do so, we are actually dealing with representatives of such elements in $\widetilde \DD(X,A)$. Thus, for instance, when we write $x\in \sigma$, for $\sigma\in \DD(X,A)$, or consider bijections $\phi\colon \sigma\to\tau$, for $\sigma,\tau\in \DD(X,A)$, we mean  $x\in\widetilde\sigma$ and $\phi\colon \widetilde\sigma\to\widetilde\tau$ is a bijection for suitable representatives $\widetilde\sigma,\widetilde\tau\in \widetilde\DD(X)$. Observe that the spaces $\DD_p(X,A)$ and $\DD_p(X/A,\{A\})$ are naturally isometric for any $p\in [1,\infty]$.


\section{Metrizability: Proof of Theorem~\ref{thm:Di-metric}}
\label{s:metrizability}

Theorem~\ref{thm:Di-metric} follows from the following two lemmas. Throughout, we let $(X,A)$ be a metric pair.


\begin{lem} \label{lem:Di-metric}
If $X \setminus A$ is discrete, then $d_\infty(\sigma,\tau) \neq 0$ for any $\sigma,\tau \in \Di{X,A}$ with $\sigma \neq \tau$.
\end{lem}


\begin{proof}
Since $\sigma \neq \tau$, there exists a point $x \in X \setminus A$ such that if $m_\sigma,m_\tau \in \NN \cup \{\aleph_0\}$ are the multiplicities of $x$ in $\sigma,\tau$ (respectively), then $m_\sigma \neq m_\tau$; without loss of generality, $m_\sigma > m_\tau$. As $X \setminus A$ is discrete, $\{x\}$ is open in $X \setminus A$, implying that $B_\varepsilon(x) \subseteq \{x\} \cup A$ for some $\varepsilon > 0$. As $x \notin A$ and $A$ is closed, we have $d(x,A) > 0$ and therefore we may assume (without loss of generality) that $\varepsilon \leq d(x,A)$; it then follows that $B_\varepsilon(x) = \{x\}$.

We claim that $d_\infty(\sigma,\tau) \geq \varepsilon$. Indeed, as $m_\sigma > m_\tau$, any bijection $\gamma\colon \sigma \to \tau$ must send a copy of $x$ in $\sigma$ to a point $y$ in $\tau$ such that $x \neq y$ in $X$, and therefore $d(x,y) \geq \varepsilon$. Thus $\sup_{z \in \sigma} d(z,\gamma(z)) \geq \varepsilon$; taking the infimum over all bijections $\gamma$ then gives $d_\infty(\sigma,\tau) \geq \varepsilon > 0$, as claimed.
\end{proof}


\begin{lem}
If $X \setminus A$ is not discrete, then there exist $\sigma,\tau \in \Di{X,A}$ such that $\sigma \neq \tau$ and $d_\infty(\sigma,\tau) = 0$.
\end{lem}


\begin{proof}
Since $X \setminus A$ is not discrete, there exists a point $x \in X \setminus A$ such that $\{x\}$ is not open in $X \setminus A$. It follows that there exists a sequence $(x_n)_{n=0}^\infty$ in $X$ such that $x_n \neq x$ for any $n$ but $x_n \to x$ as $n \to \infty$. Now let $\sigma = \{\!\{ x_n : n \in \NN \}\!\}$ and $\tau = \sigma \sqcup \{\!\{ x \}\!\}$. Then clearly $\sigma,\tau \in \Di{X,A}$ (as the sequence $(d(x_n,A))_{n=0}^\infty$ in $\R$ converges to $d(x,A)$ and so is bounded), and $\sigma \neq \tau$ since $x \notin \sigma$ but $x \in \tau$.

Now let $\varepsilon > 0$, and let $N = N_\varepsilon \in \NN$ be such that $d(x_n,x) \leq \varepsilon/2$ for any $n \geq N$. Define a bijection $\gamma_\varepsilon\colon \sigma \to \tau$ by $\gamma_\varepsilon(x_n) = x_n$ for $n < N$, $\gamma_\varepsilon(x_N) = x$, and $\gamma_\varepsilon(x_n) = x_{n-1}$ for $n > N$. We then have $d(x_n,\gamma_\varepsilon(x_n)) = 0$ for $n < N$, $d(x_N,\gamma_\varepsilon(x_N)) = d(x_N,x) \leq \varepsilon/2$, and $d(x_n,\gamma_\varepsilon(x_n)) = d(x_n,x_{n-1}) \leq d(x_n,x)+d(x_{n-1},x) \leq \varepsilon$ for $n > N$. This implies that $\sup_{y \in \sigma} d(y,\gamma_\varepsilon(y)) \leq \varepsilon$, and therefore $d_\infty(\sigma,\tau) \leq \varepsilon$. As $\varepsilon$ was arbitrary, it follows that $d_\infty(\sigma,\tau) = 0$, as required.
\end{proof}


\section{Completeness: Proofs of Theorem~\ref{thm:Di-complete} and Corollary~\ref{cor:Di-complete}} 
\label{s:completeness}

Theorem~\ref{thm:Di-complete} follows from the following two lemmas.

\begin{lem}
\label{lem:Dp.complete.then.X.complete}
For any $p\in [1,\infty]$, if $\DD_p(X,\{a_0\})$ is complete, then so is $X$.
\end{lem}

\begin{proof}
Let $(x_n)_{n=0}^\infty$ be a Cauchy sequence in $X$. Then $(d(x_n,a_0))_{n=0}^\infty$ is a Cauchy sequence in $\R$, and so converges. If $d(x_n,a_0) \to 0$ as $n \to \infty$ then $x_n \to a_0$ and so the sequence $(x_n)$ converges. Thus, we may assume that $d(x_n,a_0) \to \delta$ as $n \to \infty$ for some $\delta > 0$.

Now for each $n \in \NN$ let $\sigma_n = \mset{ x_n} \in \DD_p(X,\{a_0\})$. For each $n,m \in \NN$ there is a bijection $\sigma_n \to \sigma_m$ sending $x_n \mapsto x_m$, implying that $d_p(\sigma_n,\sigma_m) \leq d(x_n,x_m)$. Thus, $(\sigma_n)_{n=0}^\infty$ is a Cauchy sequence in $\DD_p(X,\{a_0\})$ and so converges to some $\sigma \in \DD_p(X,\{a_0\})$. Now let $\varepsilon \in (0,\delta/2]$. Then there exists $N_\varepsilon \in \NN$ such that $d_p(\sigma_n,\sigma) < \varepsilon$, and so there exists a bijection $\gamma_n\colon \sigma_n \to \sigma$ such that $d(x',\gamma_n(x')) < \varepsilon \leq \delta/2$ for every $x' \in \sigma_n$. This implies that $\sigma$ contains a unique point $x \in X \setminus B_{\delta/2}(a_0)$ and that $\gamma_n(x_n) = x$, and hence $d(x_n,x) < \varepsilon$, for every $n \geq N_\varepsilon$. Therefore, the sequence $(x_n)$ converges to $x \in X$, as required.
\end{proof}

\begin{lem}
\label{lem:X.complete.then.Dinfty.complete}Let $\ell_\infty(X,a_0)$  be the space of all sequences $(x_n)_{n=0}^\infty$ in $X$ such that $\sup_{n \in \NN} d(x_n,a_0) < \infty$, equipped with the (pseudo-)metric $\widehat{d}_\infty((x_n),(y_n)) = \sup_{n \in \NN} d(x_n,y_n)$. Let $\Phi\colon \ell_\infty(X,a_0) \to \Di{X,\{a_0\}}$ be the map sending each sequence to the (equivalence class of) the corresponding multiset. Then the following statements hold:
\begin{enumerate}
\item $\Phi$ is a continuous map.
\item If $X$ is complete then $\ell_\infty(X,a_0)$ is complete.
\item Let $(\sigma_n)_{n=0}^\infty$ be a Cauchy sequence in $\Di{X,\{a_0\}}$. Then there exists a Cauchy sequence $(\widehat\sigma_n)_{n=0}^\infty$ in $\ell_\infty(X,a_0)$ such that $\Phi(\widehat\sigma_n) = \sigma_n$ for all~$n$.
\end{enumerate}
\end{lem}

\begin{proof}
Clearly $d_\infty(\Phi(\widehat\sigma),\Phi(\widehat\tau)) \leq \widehat{d}_\infty(\widehat\sigma,\widehat\tau)$ for any $\widehat\sigma,\widehat\tau \in \ell_\infty(X,a_0)$; in particular, $\Phi$ is continuous. Moreover, it is easy to see that the arguments in the ``usual'' proof of the fact that $\ell_\infty = \ell_\infty(\R,0)$ is complete may be adapted to proof that if $X$ is complete, then so is $\ell_\infty(X,a_0)$.

Now let $(\sigma_n)_{n=0}^\infty$ be a Cauchy sequence in $\Di{X,\{a_0\}}$. Then for each $k \in \NN$ there exists $N_k \in \NN$ such that $d_\infty(\sigma_n,\sigma_m) \leq 2^{-k}$ for all $n,m \geq N_k$. In particular, for each $k \in \NN$ and $n \geq N_k$ there exists a bijection $\gamma_{k,n} \colon \sigma_{N_k} \to \sigma_n$ such that $d(x,\gamma_{k,n}(x)) \leq 2^{1-k}$ for all $x \in \sigma_{N_k}$. Without loss of generality, assume that $N_k \leq N_\ell$ whenever $k \leq \ell$ and that $N_k \to \infty$ as $k \to \infty$.

For each $n$, let $\widetilde\sigma_n$ be a countable multiset of points in $X$ representing $\sigma_n$ such that $\widetilde\sigma_n$ contains $a_0$ with infinite multiplicity. Then, for each $k \in \NN$ and $n \geq N_k$, there exists a bijection $\widetilde\gamma_{k,n} \colon \widetilde\sigma_{N_k} \to \widetilde\sigma_n$ inducing $\gamma_{k,n}$; in particular, we have $d(x,\widetilde\gamma_{k,n}(x)) \leq 2^{1-k}$ for all $x \in \widetilde\sigma_{N_k}$. We now construct a sequence $(\widehat\sigma_n)_{n=0}^\infty$ in $\ell_\infty(X,a_0)$ as follows:
\begin{itemize}
\item for $n \leq N_0$, pick an arbitrary bijection $\psi_n\colon \NN \to \widetilde\sigma_n$ and let $\widehat\sigma_n = (\psi_n(m))_{m=0}^\infty$;
\item for $n > N_0$, let $k \in \NN$ be such that $N_k < n \leq N_{k+1}$ and let 
\[
\widehat\sigma_n = \left( (\widetilde\gamma_{k,n} \circ \widetilde\gamma_{k-1,N_k} \circ \cdots \circ \widetilde\gamma_{1,N_2} \circ \widetilde\gamma_{0,N_1} \circ \psi_{N_0})(m) \right)_{m=0}^\infty;
\]
\end{itemize}
see Figure~\ref{fig:lem-complete}. It follows from the construction that $\Phi(\widehat\sigma_n) = \sigma_n$ for all~$n$. We claim that $(\widehat\sigma_n)_{n=0}^\infty$ is a Cauchy sequence.

Indeed, let $\varepsilon > 0$, and let $k \in \NN$ be such that $2^{3-k} \leq \varepsilon$. Note that for any $k',n \in \NN$ such that $N_{k'} < n \leq N_{k'+1}$ we have $\widehat{d}_\infty(\widehat\sigma_{N_{k'}},\widehat\sigma_n) \leq \sup_{x \in \widetilde\sigma_{N_{k'}}} d(x,\widetilde\gamma_{k',n}(x)) \leq 2^{1-k'}$. In particular, if in addition $k' \geq k$ then we have
\begin{align*}
\widehat{d}_\infty(\widehat\sigma_{N_k},\widehat\sigma_n) &\leq \left[ \sum_{\ell=k}^{k'-1} \widehat{d}_\infty(\widehat\sigma_{N_\ell},\widehat\sigma_{N_{\ell+1}}) \right] + \widehat{d}_\infty(\widehat\sigma_{N_{k'}},\widehat\sigma_n) \leq \sum_{\ell=k-1}^{k'-1} 2^{-\ell} \\ &< \sum_{\ell=k-1}^\infty 2^{-\ell} = 2^{2-k} \leq \frac{\varepsilon}{2}
\end{align*}
for any $n > N_k$. Therefore, for all $n,m > N_k$ we have $\widehat{d}_\infty(\widehat\sigma_n,\widehat\sigma_m) \leq \widehat{d}_\infty(\widehat\sigma_{N_k},\widehat\sigma_n) + \widehat{d}_\infty(\widehat\sigma_{N_k},\widehat\sigma_m) \leq \varepsilon$. Thus the sequence $(\widehat\sigma_n)_{n=0}^\infty$ is Cauchy, as claimed.
\end{proof}

\begin{figure}[ht]
    \centering
\begin{tikzpicture}
\draw [red,thick,->] (0,0) to[in=180] node [pos=0.6,above,rotate=5] {$\widetilde\gamma_{0,N_0+1}$} (6.9,2);
\draw [red,thick,->] (0,0) to node [pos=0.6,above] {$\widetilde\gamma_{0,N_1}$} (6.9,0);
\draw [red!80!blue,thick,->] (7,0) to[in=180] node [pos=0.6,above,rotate=20] {$\widetilde\gamma_{1,N_1+1}$} (10.9,2);
\draw [red!80!blue,thick,->] (7,0) to node [pos=0.6,above] {$\widetilde\gamma_{1,N_2}$} (10.9,0);
\draw [red!60!blue,thick,->] (11,0) to[in=180] node [pos=0.6,above,rotate=50] {$\widetilde\gamma_{2,N_2+1}$} (12.9,2);
\draw [red!60!blue,thick,->] (11,0) to node [pos=0.6,above] {$\widetilde\gamma_{2,N_3}$} (12.9,0);
\draw [red!40!blue,thick] (13,0) to[out=45,in=-120] (13.4,0.6) (13,0) -- (13.4,0);
\draw [red!40!blue,thick,densely dotted] (13.4,0.6) to[out=60,in=-150] (13.8,1) (13.4,0) -- (13.8,0);
\draw [green!70!black,ultra thick,loosely dotted] (0,0.28) -- (0,1.72) (7,0.28) -- (7,1.72) (11,0.28) -- (11,1.72) (13,0.28) -- (13,1.72);
\draw [fill=green!70!black] (0,2) circle (2pt) node [above] {$x_0 = \psi_0(m)$};
\draw [fill=green!70!black] (0,0) circle (2pt) node [below] {$x_{N_0} = \psi_{N_0}(m)$};
\draw [fill=green!70!black] (7,2) circle (2pt) node [above] {$x_{N_0+1}$};
\draw [fill=green!70!black] (7,0) circle (2pt) node [below,yshift=-2pt] {$x_{N_1}$};
\draw [fill=green!70!black] (11,2) circle (2pt) node [above] {$x_{N_1+1}$};
\draw [fill=green!70!black] (11,0) circle (2pt) node [below,yshift=-2pt] {$x_{N_2}$};
\draw [fill=green!70!black] (13,2) circle (2pt) node [above] {$x_{N_2+1}$};
\draw [fill=green!70!black] (13,0) circle (2pt) node [below,yshift=-2pt] {$x_{N_3}$};
\draw (0,-0.7) -- (0,-1.3) (7,-0.7) -- (7,-1.3) (11,-0.7) -- (11,-1.3) (13,-0.7) -- (13,-1.3);
\draw [red,<->] (0.05,-1) -- (6.95,-1) node [midway,below] {$\leq 2$};
\draw [red!80!blue,<->] (7.05,-1) -- (10.95,-1) node [midway,below] {$\leq 1$};
\draw [red!60!blue,<->] (11.05,-1) -- (12.95,-1) node [midway,below,yshift=2pt] {$\leq \frac{1}{2}$};
\end{tikzpicture}
\caption{The proof of Lemma~\ref{lem:X.complete.then.Dinfty.complete}: construction of the sequence $(x_n)_{n=0}^\infty = \big((\widehat\sigma_n)_m\big)_{n=0}^\infty$ of points in $X$ for a fixed $m \in \NN$.}
    \label{fig:lem-complete}
\end{figure}
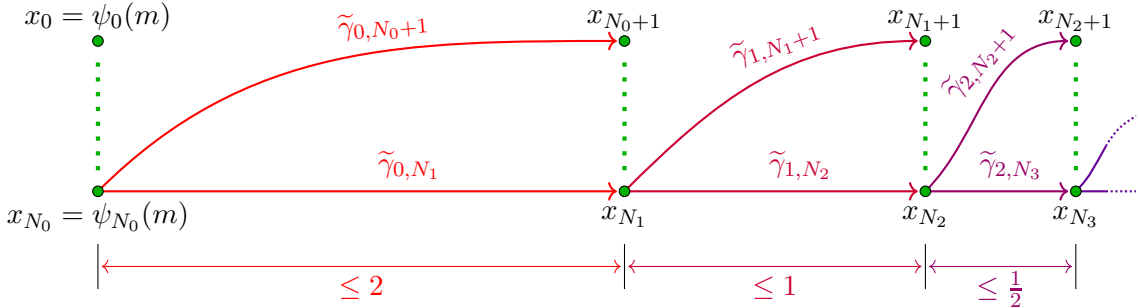

\begin{proof}[Proof of Theorem~\ref{thm:Di-complete}]
Let $p\in [1,\infty]$ and let $(X,A)$ be a metric pair.
In view of the isometry $\DD_p(X,A) \cong \DD_p(X/A,\{A\})$, $p\in [1,\infty]$, Theorem~\ref{thm:Di-complete} is equivalent to the statement that $\DD_p(X,\{a_0\})$ is complete if and only if $X$ is complete, where $a_0 \in X$ is any point. Thus, the ``only if'' implication of Theorem~\ref{thm:Di-complete} follows from Lemma~\ref{lem:Dp.complete.then.X.complete}. Suppose now that $X$ is complete. If $1\leq p<\infty$, then $\DD_p(X,A)$ is complete by \cite[Theorem B (1)]{CGGGMS}. If $p=\infty$, then it follows from Lemma~\ref{lem:X.complete.then.Dinfty.complete} that $\DD_\infty(X,A)$ is complete. Thus, $\DD_p(X,A)$ is complete for all $p\in [1,\infty]$.
\end{proof}

We now prove Corollary~\ref{cor:Di-complete}.

\begin{proof}[Proof of Corollary~\ref{cor:Di-complete}]
By Theorem~\ref{thm:Di-complete}, it is enough to show that if $X$ is complete then so is $X/A$. Thus, let $(x_n)_{n=0}^\infty$ be a Cauchy sequence in $X/A$. Then $(d_{X/A}(x_n,A))_{n=0}^\infty$ is a Cauchy sequence in $\R$, and so converges. If $d_{X/A}(x_n,A) \to 0$ as $n \to \infty$ then $x_n \to A$. Otherwise, $d_{X/A}(x_n,A) \to \delta$ as $n \to \infty$ for some $\delta > 0$. As $(x_n)$ is Cauchy, it follows that for each $\varepsilon \in (0,\delta/2)$ there exists $N_\varepsilon$ such that for all $n,m \geq N_\varepsilon$ we have $d_{X/A}(x_n,x_m) \leq \varepsilon < \delta/2$, implying that $x_n,x_m \notin B_{\delta/2}(A)$ (in particular, $x_n,x_m \neq A$) and $d_X(x_n,x_m) = d_{X/A}(x_n,x_m)$. In particular, $(x_n)$ is also a Cauchy sequence in $X$, and so converges to some $x \in X \setminus A$; but then we also have $x_n \to x$ in $X/A$ as $n \to \infty$.
\end{proof}


\section{Separability: Proof of Theorem~\ref{thm:Di-separable}}
\label{s:separability}

Theorem~\ref{thm:Di-separable} follows from the following two results.


\begin{lem}
If $B_D(A) \setminus B_\delta(A)$ is totally bounded for all $D > \delta > 0$, then $\Di{X,A}$ is separable.
\end{lem}

\begin{proof}
Let $n \in \NN_{\geq 1}$. Then $B_n(A) \setminus B_{1/n}(A)$ is totally bounded, and so there exist $x_1^{(n)},\ldots,x_{k_n}^{(n)} \in X$ such that $B_n(A) \setminus B_{1/n}(A) \subseteq \bigcup_{m=1}^{k_n} B_{1/n}\left(x_m^{(n)}\right)$. For $m \in \{1,\ldots,k_n\}$, we define $U_m^{(n)}$ inductively as
\[
U_m^{(n)} = \left[ B_n(A) \cap B_{1/n}\left(x_m^{(n)}\right) \right] \setminus \left[ B_{1/n}(A) \cap B_{1/n}\left(x_m^{(n)}\right) \cap \bigcup_{\ell=1}^{m-1} U_\ell^{(n)} \right].
\]
We thus get a collection of pairwise disjoint subsets $U_1^{(n)},\ldots,U_{k_n}^{(n)} \subseteq X$ such that $B_n(A) \setminus B_{1/n}(A) = \bigcup_{m=1}^{k_n} U_m^{(n)}$ and $U_m^{(n)} \subseteq B_{1/n}\left(x_m^{(n)}\right)$.

Now let $\mathcal{F}_n \subseteq \Di{X,A}$ be the set of diagrams in which each point (outside $A$) is equal to $x_m^{(n)}$ for some $m$. For each $m$, the multiplicity of $x_m^{(n)}$ is in $\NN \cup \{\aleph_0\}$, leading to $\left| \NN \cup \{\aleph_0\} \right| = \aleph_0$ choices. Thus the cardinality of $\mathcal{F}_n$ is $\aleph_0^{k_n} = \aleph_0$. Now let $\mathcal{F} = \bigcup_{n=1}^\infty \mathcal{F}_n$, so that $\mathcal{F}$ is a countable union of countable sets and thus countable. We claim that $\mathcal{F}$ is dense in $\Di{X,A}$.

Indeed, let $\sigma \in \Di{X,A}$, and let $\varepsilon > 0$. Pick $n \in \NN$ such that $n > d_\infty(\sigma,\sigma_\varnothing)$ and $\frac{1}{n} \leq \varepsilon$. Let $\tau \in \mathcal{F}_n \subseteq \mathcal{F}$ be such that the multiplicity of $x_m^{(n)}$ in $\tau$ is equal to the number of points of $\sigma$ lying inside $U_m^{(n)}$. Then $d_\infty(\sigma,\tau) \leq \frac{1}{n} \leq \varepsilon$: indeed, there is a bijection $\sigma \to \tau$ sending any point in $U_m^{(n)}$ to $x_m^{(n)}$ and any point in $B_{1/n}(A)$ to a point in $A$ (this leaves no points of $\sigma$ unmatched as $d_\infty(\sigma,\sigma_\varnothing) < n$ implies that all points of $\sigma$ lie in $B_n(A)$). Therefore, $\mathcal{F}$ is dense in $\Di{X,A}$, as claimed.
\end{proof}


\begin{lem}
If there exists $D > \delta > 0$ such that $B_D(A) \setminus B_\delta(A)$ is not totally bounded, then $\Di{X,A}$ is not separable.
\end{lem}

\begin{proof}
Since $B_D(A) \setminus B_\delta(A)$ is not totally bounded, there exists $\varepsilon > 0$ such that $B_D(A) \setminus B_\delta(A)$ cannot be covered by finitely many $\varepsilon$-balls; without loss of generality, assume that $\varepsilon \leq \delta$. Now let $\mathcal{X} \subseteq B_D(A) \setminus B_\delta(A)$ be a maximal subset such that $d(x,y) \geq \varepsilon$ for all $x,y \in \mathcal{X}$ with $x \neq y$ (such an $\mathcal{X}$ exists by Zorn's Lemma). Then $\varepsilon$-balls around the points of $\mathcal{X}$ cover $B_D(A) \setminus B_\delta(A)$, implying that $\mathcal{X}$ is infinite. In particular, there exists a collection $\{ x_i : i \in \NN \}$ of points in $B_D(A) \setminus B_\delta(A)$ such that $d(x_i,x_j) \geq \varepsilon$ whenever $i \neq j$.

Let $\{ \sigma_i : i \in \NN \}$ be a countable subset of $\Di{X,A}$. Define a diagram $\tau$ by
\[
\tau = \{\!\{ x_i : d(x,x_i) \geq \varepsilon/2 \text{ for all } x \in \sigma_i \}\!\};
\]
since $d(x_i,A) < D$ for all $i$ we have $\tau \in \Di{X,A}$. Note that since $d(x_i,A) \geq \delta > \varepsilon/2$ for each $i$, the definition of $\tau$ is independent of the choice of representative of $\sigma_i$ (i.e.\ adding/removing points in $A$ to/from $\sigma_i$ does not change $\tau$). We now claim that $d_\infty(\tau,\sigma_i) \geq \varepsilon/2$ for each $i$, so that $\tau$ is not in the closure of $\{ \sigma_i : i \in \NN \}$ in $\Di{X,A}$; this will imply that $\{ \sigma_i : i \in \NN \}$ is not dense in $\Di{X,A}$, as required.

To prove our claim, let $i \in \NN$ and let $\gamma\colon \tau \to \sigma_i$ be a bijection.
\begin{description}
\item[Case 1] Suppose that $d(x,x_i) \geq \varepsilon/2$ for all $x \in \sigma_i$, and so $x_i \in \tau$. We then have $d(x_i,x) \geq \varepsilon/2$ for all $x \in \sigma_i$ with $x \notin A$, and $d(x_i,A) \geq \delta \geq \varepsilon > \varepsilon/2$ since $x_i \notin B_\delta(A)$. Thus we must have $d(x_i,\gamma(x_i)) \geq \varepsilon/2$.
\item[Case 2] Suppose that $d(x,x_i) < \varepsilon/2$ for some $x \in \sigma_i$, and so $x_i \notin \tau$. Then every point of $\tau$ is either in $A$, or equal to $x_j$ for some $j \neq i$. But we have $d(x,A) \geq d(x_i,A) - d(x,x_i) > \delta - \varepsilon/2 \geq \varepsilon/2$ since $\delta \geq \varepsilon$, and $d(x,x_j) \geq d(x_i,x_j) - d(x,x_i) > \varepsilon-\varepsilon/2 = \varepsilon/2$. Therefore, we must have $d(\gamma^{-1}(x),x) > \varepsilon/2$.
\end{description}
It follows that in either case there exists $y \in \tau$ such that $d(y,\gamma(y)) \geq \varepsilon/2$, and in particular $\sup_{y \in \tau} d(y,\gamma(y)) \geq \varepsilon/2$. Taking the infimum over all bijections $\gamma\colon \tau \to \sigma_i$ gives $d_\infty(\tau,\sigma_i) \geq \varepsilon/2$, as claimed.
\end{proof}

\begin{example} \label{ex:usual-not-separable}
Consider $\Di{\R^2_{\geq 0},\Delta}$, the usual space of persistence diagrams. Then for any $D > \delta > 0$, the ``strip'' $B_D(\Delta) \setminus B_\delta(\Delta)$ is unbounded and, hence, not totally bounded. It follows that $\Di{\R^2_{\geq 0},\Delta}$ is not separable.
\end{example}

\begin{example}
Consider $\Di{[0,\infty),\{0\}}$. Then, for any $D > \delta > 0$, the set $B_D(\{0\}) \setminus B_\delta(\{0\})$ is just the interval $[\delta,D)$, which is clearly totally bounded. Thus, $\Di{[0,\infty),\{0\}}$ is separable.
\end{example}


\section{Geodesicity: Proof of Theorem \ref{thm:Di-geodesic}} \label{s:geodesicity}

In the proof of Theorem~\ref{thm:Di-geodesic} we use the concepts of ultrafilters and ultralimits, defined as follows.

\begin{defn}
An \emph{ultrafilter} (\emph{on $\NN$}) is a non-zero function $\omega\colon \mathcal{P}(\NN) \to \{0,1\}$, where $\mathcal{P}(\NN)$ denotes the power set of $\NN$, that is \emph{finitely additive}---that is, $\omega(A) + \omega(B) = \omega(A \cup B)$ for any disjoint subsets $A,B \subseteq \NN$. A \emph{principal ultrafilter} is an indicator function $\mathbbm{1}_n$ (for some $n \in \NN$), defined by setting $\mathbbm{1}_n(A) = 1$ if and only if $n \in A$. Given a sequence $(x_n)_{n=1}^\infty$ in a metric space $X$ and an ultrafilter $\omega$, an \emph{ultralimit} $\lim_\omega x_n$ of $(x_n)$ with respect to $\omega$ is a point $x \in X$ such that $\omega(\{n \in \NN : d(x,x_n) < \varepsilon \}) = 1$ for every $\varepsilon > 0$.
\end{defn}

It is easy to check, using finite additivity, that an ultrafilter $\omega$ is non-principal if and only if $\omega(F) = 0$ for every finite subset $F \subset \NN$. The existence of a non-principal ultrafilter is guaranteed by the axiom of choice \cite[\S3.1]{KL}. We refer the interested reader to \cite{Goldbring} for an introduction to the theory of ultrafilters and their use throughout mathematics.

In any metric space $X$, an ultralimit of a sequence $(x_n)$ is always unique (when it exists): indeed, if $x \neq y$ were two different ultralimits of $(x_n)$ then $\omega$ would send both $\{n \in \NN : d(x,x_n) < \varepsilon \}$ and $\{n \in \NN : d(y,x_n) < \varepsilon \}$ to $1$, where $\varepsilon = \frac{1}{2}d(x,y)$, contradicting finite additivity of $\omega$. Moreover, if $X$ is compact then such an ultralimit will always exist \cite[\S3.1]{KL}. If $\omega$ is a non-principal ultrafilter, then clearly $\lim_\omega x_n$ will be an accumulation point of the sequence $(x_n)$.

Our main use of ultrafilters and ultralimits is based on their ability to pick out accumulation points of sequences \emph{consistently}. For instance, roughly speaking, in the proof of Proposition~\ref{prop:Di-geodesic} below we aim to construct a bijection $\widetilde\gamma\colon \widetilde\sigma \to \widetilde\tau$ as a `limit' of a sequence of bijections $\widetilde\gamma_n\colon \widetilde\sigma \to \widetilde\tau$. Given a point $x \in \widetilde\sigma$ it may turn out, for example, that there exists a subset $Y_x \subseteq \widetilde\tau$ such that $\{ n \in \NN : \widetilde\gamma_n(x) = y \}$ is infinite for all $y \in Y_x$; analogously, given $y \in \widetilde\tau$, there may exist $X_y \subseteq \widetilde\sigma$ such that $\{ n \in \NN : \widetilde\gamma_n^{-1}(y) = x \}$ is infinite for all $x \in X_y$. It is then unclear which (if any) point of $Y_x$ (respectively $X_y$) should be declared to be $\widetilde\gamma(x)$ (respectively $\widetilde\gamma^{-1}(y)$), doing this procedure for all $x \in \widetilde\sigma$ and $y \in \widetilde\tau$ in a consistent way: that is, if $y \in Y_x$ is declared to be $\widetilde\gamma(x)$ then $x \in X_y$ needs to be declared to be $\widetilde\gamma^{-1}(y)$, and vice versa. Using ultralimits with respect to a non-principal ultrafilter $\omega$ allows us to circumvent such an issue---in particular, given $x \in X$ at most one of the subsets $\{ n \in \NN : \widetilde\gamma_n(x) = y \}$ for $y \in Y_x$ is sent to $1$ by $\omega$, allowing us to pick out the `correct' $y \in Y_x$ unambiguously.

The main ingredient in the proof of Theorem~\ref{thm:Di-geodesic} is the following result.

\begin{prop} \label{prop:Di-geodesic}
Let $X$ be a proper metric space and let $\sigma,\tau \in \Di{X,A}$. Then there exist diagrams $\sigma',\tau' \in \Di{X,A}$ such that $d_\infty(\sigma,\sigma') = d_\infty(\tau,\tau') = 0$ and a bijection $\gamma\colon \sigma' \to \tau'$ such that $d_\infty(\sigma,\tau) = \sup \{ d(x,\gamma(x)) : x \in \sigma' \}$.
\end{prop}

\begin{proof}
By the definition of the pseudometric $d_\infty$, there is a sequence of bijections $(\gamma_n\colon \sigma \to \tau)_{n=1}^\infty$ such that $d(x,\gamma_n(x)) \leq d_\infty(\sigma,\tau) + \frac{1}{n}$ for all $n$ and all $x \in \sigma$. We can lift each $\gamma_n$ to a bijection $\widetilde\gamma_n \colon \widetilde\sigma_n \to \widetilde\tau_n$ for some countable multisets $\widetilde\sigma_n$ and $\widetilde\tau_n$ in $X$ representing diagrams $\sigma$ and $\tau$, respectively. Now let $\widetilde\sigma$ (respectively $\widetilde\tau$) be the multiset in $X$ consisting of the multiset $\{\!\{ x \in \sigma : x \notin A \}\!\}$ (respectively $\{\!\{ x \in \tau : x \notin A \}\!\}$) together with a countably infinite number of copies of each point $a \in A$ appearing in $\widetilde\sigma_n \cup \widetilde\tau_n$ for some $n \geq 1$. It is easy to see that $\widetilde\sigma$ and $\widetilde\tau$ are countable. Moreover, since any point $a \in A$ appears in $\widetilde\sigma$ and in $\widetilde\tau$ with the same multiplicity (zero or infinite), with $a$ having infinite multiplicity whenever it appears in $\widetilde\sigma_n$ or in $\widetilde\tau_n$, it follows that the bijection $\widetilde\gamma_n \colon \widetilde\sigma_n \to \widetilde\tau_n$ can be extended to a bijection $\widetilde\gamma_n \colon \widetilde\sigma \to \widetilde\tau$ such that $\widetilde\gamma_n(x) = x$ in $X$ whenever $x \notin \widetilde\sigma_n$. In particular, such an extended bijection $\widetilde\gamma_n\colon \widetilde\sigma \to \widetilde\tau$ satisfies $d(x,\widetilde\gamma_n(x)) \leq d_\infty(\sigma,\tau)+\frac{1}{n}$ for all $x \in \widetilde\sigma$.

Now let $\omega\colon \mathcal{P}(\NN) \to \{0,1\}$ be a non-principal ultrafilter. We call a point $x \in \widetilde\sigma$ \emph{convergent} if there exists $y_x \in \widetilde\tau$ such that $\widetilde\gamma_n(x) = y_x$ $\omega$-almost surely (that is, $\omega$ sends $\{ n \in \NN : \widetilde\gamma_n(x) = y_x \}$ to $1$), and we let $\widetilde\sigma_c \subseteq \widetilde\sigma$ be the multiset of convergent points in $\widetilde\sigma$. Similarly, we call a point $y \in \widetilde\tau$ \emph{convergent} if $\widetilde\gamma_n^{-1}(y) = x_y$ $\omega$-almost surely for some $x_y \in \widetilde\sigma$, and we let $\widetilde\tau_c \subseteq \widetilde\tau$ be the multiset of convergent points in $\widetilde\tau$. Note that given any $x \in \widetilde\sigma$ and $y \in \widetilde\tau$, we have $x = x_y$ if and only if $y = y_x$.

Given  $x \in \widetilde\sigma \setminus \widetilde\sigma_c$, note that $B_{d_\infty(\sigma,\tau)+1}(x)$ is compact (as $X$ is proper), and so the sequence $(\widetilde\gamma_n(x))_{n=1}^\infty \subseteq B_{d_\infty(\sigma,\tau)+1}(x)$ has a (unique) ultralimit $y_x^\omega := \lim_\omega \widetilde\gamma_n(x)$ in $X$. Similarly, given $y \in \widetilde\tau \setminus \widetilde\tau_c$ we can define $x_y^\omega := \lim_\omega \widetilde\gamma_n^{-1}(y)$; see Figure~\ref{fig:geodesic-<=}. Now let $\widetilde\sigma' := \widetilde\sigma \sqcup \{\!\{ x_y^\omega : y \in \widetilde\tau \setminus \widetilde\tau_c \}\!\}$ and $\widetilde\tau' := \widetilde\tau \sqcup \{\!\{ y_x^\omega : x \in \widetilde\sigma \setminus \widetilde\sigma_c \}\!\}$. We define a mapping $\widetilde\gamma\colon \widetilde\sigma' \to \widetilde\tau'$ by setting $\widetilde\gamma(x) = y_x$ for $x \in \widetilde\sigma_c$, $\widetilde\gamma(x) = y_x^\omega$ for $x \in \widetilde\sigma \setminus \widetilde\sigma_c$, and $\widetilde\gamma(x_y^\omega) = y$ for $y \in \widetilde\tau \setminus \widetilde\tau_c$. Finally, let $\sigma',\tau' \in \Di{X,A}$ be the diagrams represented by $\widetilde\sigma',\widetilde\tau'$ (respectively), and let $\gamma\colon \sigma' \to \tau'$ be the mapping induced by $\widetilde\gamma$.

Note that $\widetilde\gamma$ (and so $\gamma$) is a well-defined bijection by construction. Moreover, for any $x \in \widetilde\sigma'$ we have $d(x,\widetilde\gamma(x)) \leq d_\infty(\sigma,\tau)$: indeed, we have
\begin{gather*}
d(x,\widetilde\gamma(x)) = \lim_\omega d(x,\widetilde\gamma_n(x)) \leq \limsup_{n \to \infty} d(x,\widetilde\gamma_n(x)) \leq d_\infty(\sigma,\tau)
\shortintertext{if $x \in \widetilde\sigma$, and}
d(x,\widetilde\gamma(x)) = d(y,x_y^\omega) = \lim_\omega d(y,\widetilde\gamma_n^{-1}(y)) \leq \limsup_{n \to \infty} d(y,\widetilde\gamma_n^{-1}(y)) \leq d_\infty(\sigma,\tau)
\end{gather*}
if $x = x_y^\omega$ for some $y \in \widetilde\tau \setminus \widetilde\tau_c$; in both cases, the first inequality follows from the fact that given a sequence $(r_n)_{n=1}^\infty$ in the interval $[0,d_\infty(\sigma,\tau)+1]$, the ultralimit $\lim_\omega r_n$ is an accumulation point of~$(r_n)$. This implies that $\sup \{ d(x,\gamma(x)) : x \in \sigma' \} = \sup \{ d(x,\widetilde\gamma(x)) : x \in \widetilde\sigma' \} \leq d_\infty(\sigma,\tau)$. Moreover, by the definition of $d_\infty$ we have $\sup \{ d(x,\gamma(x)) : x \in \sigma' \} \geq d_\infty(\sigma',\tau')$. Since $|d_\infty(\sigma,\tau) - d_\infty(\sigma',\tau')| \leq d_\infty(\sigma,\sigma') + d_\infty(\tau,\tau')$, it remains to show that $d_\infty(\sigma,\sigma') = d_\infty(\tau,\tau') = 0$.

We now claim that $d_\infty(\sigma,\sigma') = 0$ (the proof that $d_\infty(\tau,\tau') = 0$ is analogous). Indeed, let $\varepsilon > 0$, and let $N \in \NN$ be such that $\frac{2}{N} \leq \varepsilon$. We will find a bijection $\widetilde\delta_N \colon \widetilde\sigma \to \widetilde\sigma'$ such that $d(x,\widetilde\delta_N(x)) \leq \frac{2}{N} \leq \varepsilon$ for any $x \in \widetilde\sigma$.

Let $y \in \widetilde\tau \setminus \widetilde\tau_c$. We claim that there exists a sequence $(x_y^{(n)})_{n=1}^\infty$ of points in $\widetilde\sigma$ such that $x_y^{(n)} \to x_y^\omega$ as $n \to \infty$ and $x_y^{(n)} \neq x_y^{(m)}$ (as elements of $\widetilde\sigma$) when $n \neq m$. Indeed, in order to show this it is enough to show that for each $\varepsilon > 0$, the submultiset $X_\varepsilon = \{\!\{ x \in \widetilde\sigma : d(x,x_y^\omega) < \varepsilon \}\!\}$ of $\widetilde\sigma$ is infinite. Suppose for contradiction that this is not the case, and so $X_\varepsilon = \{\!\{ x_1,\ldots,x_k \}\!\}$ for some $\varepsilon > 0$, and for $1 \leq i \leq k$ let $A_i = \{ n \in \NN : \widetilde\gamma_n^{-1}(y) = x_i \}$. Then $\omega(A_i) = 0$ for each $i$ since $y \notin \widetilde\tau_c$; however, we have $\omega(A_1 \cup \cdots \cup A_k) = \omega(\{ n \in \NN : d(\widetilde\gamma_n^{-1}(y),x_y^\omega) < \varepsilon \}) = 1$ by the definition of $x_y^\omega$. This contradicts the fact that $\omega$ is finitely additive; thus $X_\varepsilon$ must be infinite, and so a sequence $(x_y^{(n)})$ with the claimed properties does exist.

Moreover, we can pass each of the sequences $(x_y^{(n)})_{n=1}^\infty$ for $y \in \widetilde\tau \setminus \widetilde\tau_c$ to a subsequence so that $x_y^{(n)} \neq x_{y'}^{(m)}$ (as elements of $\widetilde\sigma$) whenever $(y,n) \neq (y',m)$: for instance, we may pick an enumeration $\widetilde\tau \setminus \widetilde\tau_c = \{\!\{ y_1,y_2,\ldots \}\!\}$ and inductively build the collection of tuples $\{\!\{ (x_{y_i}^{(1)},\ldots,x_{y_i}^{(n)}) : 1 \leq i \leq n \}\!\}$ with no repeating elements. Finally, after passing each of the sequences $(x_y^{(n)})_{n=1}^\infty$ to a subsequence once again, if necessary, we may assume that $d(x_y^{(n)},x_y^\omega) \leq \frac{1}{n}$ for all $y$ and $n$.

We now define the bijection $\widetilde\delta_N\colon \widetilde\sigma \to \widetilde\sigma'$ as follows: given $x \in \widetilde\sigma$,
\begin{itemize}
\item if $x = x_y^{(N)}$ for some $y \in \widetilde\tau \setminus \widetilde\tau_c$, let $\widetilde\delta_N(x) = x_y^\omega$;
\item if $x = x_y^{(n)}$ for some $y \in \widetilde\tau \setminus \widetilde\tau_c$ and some $n \geq N+1$, let $\widetilde\delta_N(x) = x_y^{(n-1)}$;
\item otherwise, let $\widetilde\delta_N(x) = x$;
\end{itemize}
see Figure~\ref{fig:geodesic-<=}, solid arrows. It is easy to see that $\widetilde\delta_N$ is a bijection. Moreover, we have $d(x,\widetilde\delta_N(x)) \leq \frac{2}{N}$ for all $x \in \widetilde\sigma$: indeed, we have $d(x_y^{(N)},x_y^\omega) \leq \frac{1}{N}$ and
\[
d(x_y^{(n)},x_y^{(n-1)}) \leq d(x_y^{(n)},x_y^\omega) + d(x_y^{(n-1)},x_y^\omega) \leq \frac{1}{n} + \frac{1}{n-1} \leq \frac{2}{N}
\]
for $n \geq N+1$. It follows that $d_\infty(\sigma,\sigma') \leq \sup \{ d(x,\widetilde\delta_N(x)) : x \in \widetilde\sigma \} \leq \varepsilon$, as required.
\end{proof}

\begin{figure}[ht]
    \centering
\begin{tikzpicture}
\draw [dashed,green!70!black!50,fill=green!10] (4,0) circle (2.15);
\node [green!70!black!80] at (4,-1) {$B_{1/N}(x_y^\omega)$};
\draw [thick,red,->] (2,5) to[out=0,in=-45] (2.5,5.3) to[out=135,in=60] (2.04,5.09);
\draw [thick,red,->] (3,4.2) to[out=0,in=-45] (3.5,4.5) to[out=135,in=60] (3.04,4.29);
\draw [thick,red,->] (3.5,3.5) to[out=0,in=-45] (4,3.8) to[out=135,in=60] (3.54,3.59);
\draw [thick,red,->] (4,1.8) to[out=180,in=180] (3.9,0);
\draw [thick,red,->] (4,1.2) to[out=120,in=-90] (3.85,1.5) to[out=90,in=-120] (3.94,1.71);
\draw [thick,red,->] (3.85,0.9) to[out=90,in=-120] (3.94,1.11);
\draw [thick,red,densely dotted] (3.85,0.9) to[out=-90,in=120] (3.94,0.69);
\draw [thick,blue,dash pattern={on 8pt off 2pt},->] (0,0) to[out=180,in=0] node [pos=0.6,above,blue!50] {$\leq d+1$} (-5.9,0);
\draw [thick,blue,dash pattern={on 8pt off 2pt},->] (0,0) to[out=160,in=0] node [pos=0.6,above,rotate=-10,yshift=-2pt,blue!50] {$\leq d+\frac{1}{2}$} (-4.9,1);
\draw [thick,blue,dash pattern={on 8pt off 2pt},->] (0,0) to[out=140,in=0] node [pos=0.6,above,rotate=-25,yshift=-3pt,blue!50] {$\leq d+\frac{1}{3}$} (-3.9,2);
\draw [thick,blue,dash pattern={on 8pt off 2pt},->] (0,0) to[out=90,in=-135] node [pos=0.8,left] {$\widetilde\gamma_{m_1}$} (1.93,4.93);
\draw [thick,blue,dash pattern={on 8pt off 2pt},->] (0,0) to[out=80,in=-150] node [pos=0.8,left,yshift=2pt] {$\widetilde\gamma_{m_2}$} (2.92,4.15);
\draw [thick,blue,dash pattern={on 8pt off 2pt},->] (0,0) to[out=70,in=-145] node [pos=0.8,left,yshift=3pt] {$\widetilde\gamma_{m_3}$} (3.43,3.43);
\draw [thick,blue,dash pattern={on 8pt off 2pt},->] (0,0) to[out=30,in=180] node [pos=0.7,above,rotate=20,yshift=-2pt] {$\widetilde\gamma_{m_N}$} (3.9,1.8);
\draw [thick,blue,dash pattern={on 8pt off 2pt},->] (0,0) to[out=20,in=180] node [pos=0.7,below,rotate=15,yshift=2pt] {$\widetilde\gamma_{m_{N+1}}$} (3.9,1.2);
\draw [fill=black!50] (0,0) circle (2pt) node [below] {$y$};
\draw [fill=green!70!black] (-6,0) circle (2pt) node [left] {$\widetilde\gamma_1(y)$};
\draw [fill=green!70!black] (-5,1) circle (2pt) node [left] {$\widetilde\gamma_2(y)$};
\draw [fill=green!70!black] (-4,2) circle (2pt) node [left] {$\widetilde\gamma_3(y)$};
\draw [fill=green!70!black] (2,5) circle (2pt) node [above,xshift=-4pt] {$x_y^{(1)}$};
\draw [fill=green!70!black] (3,4.2) circle (2pt) node [above,xshift=-4pt] {$x_y^{(2)}$};
\draw [fill=green!70!black] (3.5,3.5) circle (2pt) node [above,xshift=-4pt] {$x_y^{(3)}$};
\draw [green!70!black,ultra thick,loosely dotted] (3.65,3.3) to[out=-60,in=90] (4,2.1) (4,1) -- (4,0.5);
\draw [fill=green!70!black] (4,1.8) circle (2pt) node [right] {$x_y^{(N)}$};
\draw [fill=green!70!black] (4,1.2) circle (2pt) node [right] {$x_y^{(N+1)}$};
\draw [fill=green!30] (4,0) circle (2pt) node [right] {$x_y^\omega$};
\end{tikzpicture}
    \caption{The proof of Proposition~\ref{prop:Di-geodesic}: construction for a point $y \in \widetilde\tau \setminus \widetilde\tau_c$. Here $d = d_\infty(\sigma,\tau)$, and the solid arrows represent the bijection $\widetilde\delta_N\colon \widetilde\sigma \to \widetilde\sigma'$.}
    \label{fig:geodesic-<=}
\end{figure}
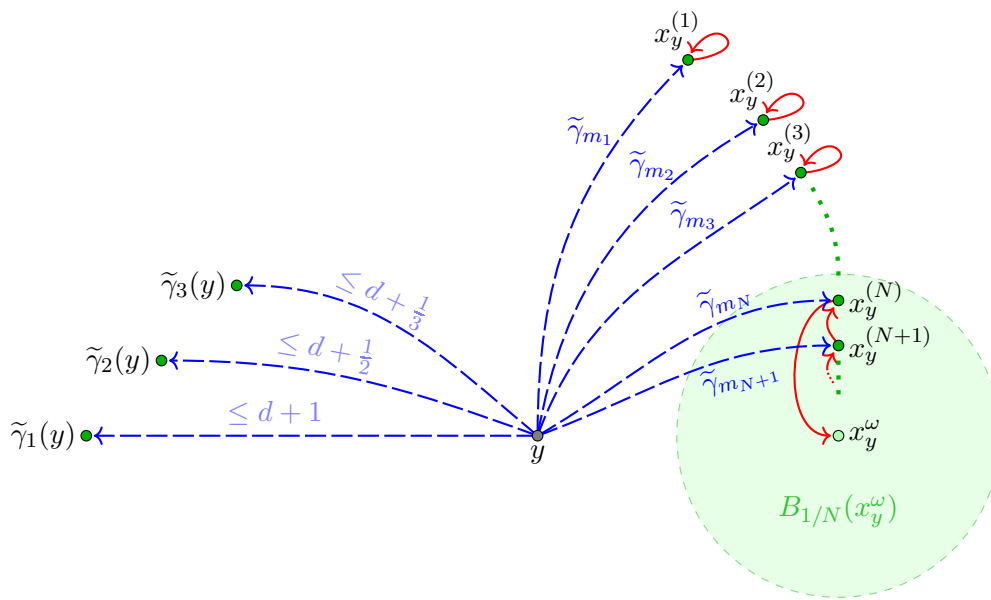

We use the following two lemmas to prove the `if' implication in Theorem~\ref{thm:Di-geodesic}.

\begin{lem} \label{lem:geod-lift}
If $X$ is a proper metric space and $A \subseteq X$ is a non-empty closed subset, then any geodesic in $X/A$ with $A$ as an endpoint can be lifted to a geodesic in $X$.
\end{lem}

\begin{proof}
Let $\eta\colon [0,1] \to X/A$ be a constant speed geodesic such that $\eta(0) = A$. Let $x = \eta(1)$ and $r = d(x,A)$; without loss of generality, suppose that $x \notin A$, and therefore $r > 0$ and $\eta(t) \notin A$ for any $t > 0$. Then for any $n \in \Z_{\geq 1}$ there exists $a_n \in A$ such that $d_X(\eta(\frac{1}{n}),a_n) \leq \frac{2r}{n}$. In particular, we have $a_n \in A \cap B$ for each $n$, where $B$ is the closed ball in $X$ with centre $x$ and radius $2r$. As $X$ is proper and $A$ is closed, $A \cap B$ is compact, implying that a subsequence of the sequence $(a_n)$ converges to some $a \in A$. We then define a map $\eta'\colon [0,1] \to X$ by $\eta'(0) = a$ and $\eta'(t) = \eta(t)$ for $t > 0$. For any $t > 0$ we then have
\[
d(a,\eta(t)) \leq \limsup_{n \to \infty} d(a_n,\eta(t)) \leq \limsup_{n \to \infty} d(a_n,\eta(1/n)) + \limsup_{n \to \infty} d(\eta(1/n),\eta(t)) = td(x,A),
\]
implying that $\eta'$ is a geodesic lifting $\eta$.
\end{proof}

\begin{lem} \label{lem:Di-geodesic-<=}
If $X$ is a proper metric space and $A \subseteq X$ is a non-empty closed subset such that any two points $x,y \in X/A$ with either $d(x,y) < d(x,A)$ or $y=A$ can be joined by a geodesic, then the pseudo-metric space $\Di{X,A}$ is geodesic.
\end{lem}

\begin{proof}
Let $\sigma,\tau \in \Di{X,A}$, and let $\sigma',\tau' \in \Di{X,A}$ and $\gamma\colon \sigma' \to \tau'$ be as given by Proposition~\ref{prop:Di-geodesic}. Without loss of generality, assume that $d(x,\gamma(x)) < \max\{d(x,A),d(\gamma(x),A)\}$ for each $x \in \sigma'$ such that $x \neq A$ and $\gamma(x) \neq A$. Then each $x \in \sigma'$ can be joined by a geodesic in $X$ to $\gamma(x) \in \tau'$ (when $x \in A$ or $\gamma(x) \in A$, this follows from Lemma~\ref{lem:geod-lift}); in particular we can pick a constant speed geodesic $\eta_x\colon [0,1] \to X$ from $x$ to $\gamma(x)$. For each $t \in [0,1]$, define a diagram $\rho_t$ as the equivalence class of the multiset $\{\!\{ \eta_x(t) : x \in \sigma' \}\!\}$; we have $\rho_t \in \Di{X,\{a_0\}}$ since $d(x,\gamma(x)) \leq d(x,a_0)+d(\gamma(x),a_0) \leq d_\infty(\sigma',\sigma_\varnothing)+d_\infty(\tau',\sigma_\varnothing)$ for each $x \in \sigma'$ and therefore
\[
d_\infty(\rho_t,\sigma_\varnothing) \leq d_\infty(\sigma',\sigma_\varnothing)+t(d_\infty(\sigma',\sigma_\varnothing)+d_\infty(\tau',\sigma_\varnothing)) \leq 2d_\infty(\sigma',\sigma_\varnothing)+d_\infty(\tau',\sigma_\varnothing) < \infty.
\]

Now since $d(\sigma,\sigma') = 0$ we have $d(\sigma,\rho) = d(\sigma',\rho)$ (similarly, $d(\tau,\rho) = d(\tau',\rho)$) for any $\rho \in \Di{X,\{a_0\}}$. Moreover, whenever $0 \leq t < u \leq 1$ we have
\[
d_\infty(\rho_t,\rho_u) \leq \sup \{ d(\eta_x(t),\eta_x(u)) : x \in \sigma' \} = (u-t)d_\infty(\sigma,\tau),
\]
and therefore the opposite inequality must also hold since
\[
d_\infty(\rho_t,\rho_u) \geq d_\infty(\sigma,\tau)-d_\infty(\sigma,\rho_t)-d_\infty(\rho_u,\tau) \geq [1-(t-0)-(1-u)] d_\infty(\sigma,\tau) = (u-t)d_\infty(\sigma,\tau).
\]
It follows that if we define $\xi\colon [0,1] \to \Di{X,\{a_0\}}$ by $\xi(0) = \sigma$, $\xi(1) = \tau$, and $\xi(t) = \rho_t$ for $0 < t < 1$, then $\xi$ is a constant speed geodesic in $\Di{X,\{a_0\}}$ from $\sigma$ to $\tau$, as required.
\end{proof}

In the other direction, we have the following.

\begin{lem} \label{lem:Di-geodesic-=>}
If $X$ and $a_0 \in X$ are such that the pseudo-metric space $\Di{X,\{a_0\}}$ is geodesic, then any two points $x,y \in X$ with either $d(x,y) < d(x,a_0)$ or $y=a_0$ can be joined by a geodesic.
\end{lem}

\begin{proof}
We say a pair $(x,y) \in X^2$ is \emph{good} if $x \neq y$ and either $d(x,y) < d(x,a_0)$ or $y=a_0$. We aim to show that any good pair of points in $X$ can be joined by a geodesic. For any good pair $(x,y)$, we pick a constant speed geodesic $\xi_{x,y}\colon [0,1] \to \Di{X,\{a_0\}}$ from $\{\!\{ x \}\!\}$ to $\{\!\{ y \}\!\}$. Note that we have $d_\infty(\{\!\{ x \}\!\},\{\!\{ y \}\!\}) = d(x,y)$; in particular, if $d(x,y) < d(x,a_0)$ then $d_\infty(\{\!\{ x \}\!\},\{\!\{ y \}\!\}) < d(x,a_0)$, and if $y=a_0$ then $d_\infty(\{\!\{ x \}\!\},\{\!\{ y \}\!\}) = d_\infty(\{\!\{x\}\!\},\sigma_\varnothing) = d(x,a_0)$, implying that we have $d_\infty(\{\!\{ x \}\!\},\{\!\{ y \}\!\}) \leq d(x,a_0)$ in either case.

The idea of the proof is to build a geodesic from $x$ to $y$ for a good pair $(x,y)$ one-quarter of the remaining distance at a time. Any point $v \in \xi_{x,y}(t)$ for $t \leq \frac{1}{4}$ satisfies either $d(v,a_0) < \frac{1}{3}d(x,a_0)$ or $d(v,a_0) > \frac{2}{3}d(x,a_0)$, and such a `gap' ensures that the partial geodesic can be built unambiguously by picking the points satisfying the latter inequality. Moreover, if such a partial geodesic from $x$ to $y$ ends at a point $u$ then one can show that the pair $(u,y)$ is again good, and so the construction can be iterated.

Let $(x,y) \in X^2$ be a good pair. Then for any $t \in [0,\frac{1}{4}]$ there exists a bijection $\gamma\colon \{\!\{ x \}\!\} \to \xi_{x,y}(t)$ such that $d(z,\gamma(z)) < \frac{1}{3} d(x,a_0)$ for every $z$ in (a chosen representative of) $\{\!\{ x \}\!\}$, implying that there exists a point $u = u(x,y,t) \in \xi_{x,y}(t)$, namely $u(x,y,t) := \gamma(x)$, such that $d(u,a_0) > \frac{2}{3}d(x,a_0)$ and $d(v,a_0) < \frac{1}{3}d(x,a_0)$ for every $v \in \xi_{x,y}(t) \setminus \{\!\{ u \}\!\}$; see Figure~\ref{fig:geodesic-=>}. Thus, as for any $0 \leq s < t \leq \frac{1}{4}$ we have $d_\infty(\xi_{x,y}(s),\xi_{x,y}(t)) = (t-s)d(x,y) < \frac{1}{3} d(x,a_0)$, any `nearly optimal' bijection $\gamma\colon \xi_{x,y}(s) \to \xi_{x,y}(t)$ must send $u(x,y,s)$ to $u(x,y,t)$ and therefore $d(u(x,y,s),u(x,y,t)) \leq (t-s) d(x,y)$.

Furthermore, we claim that $d(u,y) \leq \frac{3}{4} d(x,y)$ and that $(u,y)$ is a good pair, where $u = u(x,y,\frac{1}{4})$. If $y = a_0$, then $\{\!\{ y \}\!\} = \sigma_\varnothing$ as diagrams and the claim follows directly from the fact that $d_\infty(\xi_{x,y}(\frac{1}{4}),\sigma_\varnothing) = \frac{3}{4} d(x,y)$ (note that we cannot have $u = y$ as $d(x,u) \leq \frac{1}{4} d(x,y)$). Otherwise, we have $d(x,y) < d(x,a_0)$, and so for `nearly optimal' bijections $\gamma_x\colon \xi_{x,y}(\frac{1}{4}) \to \{\!\{ x \}\!\}$ and $\gamma_y\colon \xi_{x,y}(\frac{1}{4}) \to \{\!\{ y \}\!\}$ we have $d(\gamma_x(z),\gamma_y(z)) \leq d(z,\gamma_x(z)) + d(z,\gamma_y(z)) < d(x,a_0)$ for any $z \in \xi_{x,y}(\frac{1}{4})$. But then we must have $(\gamma_y \circ \gamma_x^{-1})(x) = y$; on the other hand, as $u$ is the only point of $\xi_{x,y}(\frac{1}{4})$ such that $d(u,a_0) > \frac{2}{3}d(x,a_0)$, for `nearly optimal' bijections we have $\gamma_x(u) = x$ and therefore $\gamma_y(u) = y$. Taking infimum over the `nearly optimal' bijections $\gamma_y$, this implies that $d(u,y) \leq \frac{3}{4} d(x,y)$. Furthermore, taking the infimum over $\gamma_x$ we have $d(x,u) \leq \frac{1}{4} d(x,y) < \frac{1}{4} d(x,a_0)$, and therefore $d(u,y) < \frac{3}{4} d(x,a_0) < d(x,a_0)-d(x,u) \leq d(u,a_0)$, implying that $(u,y)$ is a good pair, as claimed.

We now define the sequence $(u_n)_{n=0}^\infty$ inductively, by letting $u_0 = x$ and $u_n = u(u_{n-1},y,\frac{1}{4})$ for $n \geq 1$. We define a map $\eta\colon [0,1] \to X$ as follows:
\[
\eta(t) :=
\begin{cases}
u(u_n,y,1-(\frac{4}{3})^n(1-t)) &\text{if } 1-(\frac{3}{4})^n \leq t < 1-(\frac{3}{4})^{n+1},\\
y &\text{if } t=1.
\end{cases}
\]
It is straightforward to verify (using the triangle inequality) that $d(\eta(s),\eta(t)) \leq (t-s)d(x,y)$ for all $0 \leq s < t \leq 1$, and therefore that $\eta$ is a geodesic, as required.
\end{proof}

\begin{figure}[ht]
    \centering
\begin{tikzpicture}
\fill [green!5] (-3.464,-2) arc (210:150:4) -- (3.464,2) arc (30:-30:4) -- cycle;
\draw [dashed,green!70!black!50] (-3.464,-2) arc (210:150:4) (3.464,2) arc (30:-30:4);
\draw [dashed,green!70!black!50,fill=green!15] (0,0) circle (2);
\node [green!70!black!80] at (0,-1.5) {$B_{d/3}(a_0)$};
\node [green!70!black!80] at (2.6,-1.5) {$B_{2d/3}(a_0)$};
\draw [thick,blue,->] (-5.95,-0.75) to[in=-90] (-5.4,0.4);
\draw [thick,blue,->] (0,0) to[out=180,in=45] (-1.13,-0.43);
\draw [thick,blue,->] (0,0) to[out=90,in=-45] (-0.43,0.73);
\draw [thick,blue,->] (0,0) to[out=15,in=-120] (0.95,0.91);
\draw [fill=red] (0,0) circle (2pt) node [below] {$a_0$};
\draw [fill=red] (-5.95,-0.75) circle (2pt) node [below] {$x$};
\draw [fill=black] (-5.4,0.5) circle (2pt) node [above] {$u(x,y,t)$};
\draw [fill=black] (-1.2,-0.5) circle (2pt);
\draw [fill=black] (-0.5,0.8) circle (2pt);
\draw [fill=black] (1,1) circle (2pt);
\end{tikzpicture}
    \caption{An example of the situation in the proof of Lemma~\ref{lem:Di-geodesic-=>}; here $d = d(x,a_0)$, the black points represent a diagram $\xi_{x,y}(t)$ for some $t \in [0,\frac{1}{4}]$, and the arrows a `nearly optimal' bijection $\gamma\colon \{\!\{ x \}\!\} \to \xi_{x,y}(t)$.}
    \label{fig:geodesic-=>}
\end{figure}
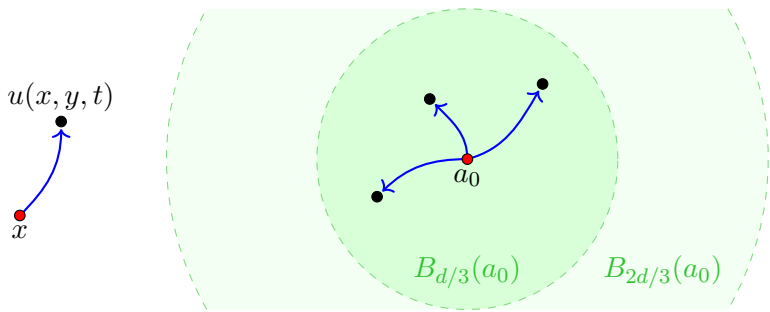

\begin{proof}[Proof of Theorem~\ref{thm:Di-geodesic}]
The `if' direction follows from Lemma~\ref{lem:Di-geodesic-<=}. The `only if' direction follows from Lemma~\ref{lem:Di-geodesic-=>} and the isometry $\Di{X,A} \cong \Di{X/A,\{A\}}$.
\end{proof}

\begin{proof}[Proof of Corollary~\ref{cor:Di-geodesic}]
In view of Theorem~\ref{thm:Di-geodesic}, it is enough to show that if $X$ is proper and geodesic, then $X/A$ is geodesic. Thus, let $x,y \in X$, and let $[x],[y] \in X/A$ be their images in $X/A$. If $d(x,y) < d(x,A)+d(y,A)$, then it is easy to verify that for any geodesic $\eta\colon [0,1] \to X$ from $x$ to $y$, the composite $q \circ \eta$ is a geodesic in $X/A$ from $[x]$ to $[y]$, where $q\colon X \to X/A$ is the quotient map. Otherwise, we have $d_{X/A}([x],[y]) = d(x,A)+d(y,A)$. But then as $X$ is proper and $A$ is closed, there exist $a,b \in A$ such that $d(x,A)=d(x,a)$ and $d(y,A)=d(y,b)$. Therefore, for any geodesics $\eta_x$ and $\eta_y$ in $X$ from $x$ to $a$ and from $b$ to $y$, respectively, the concatenation of $q \circ \eta_x$ and $q \circ \eta_y$ is a geodesic in $X/A$ from $[x]$ to $[y]$, as required.
\end{proof}

It follows from Lemma~\ref{lem:Di-geodesic-=>} that the `only if' direction of Theorem~\ref{thm:Di-geodesic} is true even without the assumption that $X$ is proper. However, the same cannot be said about the `if' direction, as the following example shows.

\begin{example} \label{ex:non-geodesic}
Let $X = c_0$ (that is, the space of all real sequences $(a_n)_{n=1}^\infty$ converging to zero, equipped with the supremum metric). Let $A = \{(0,0,\ldots)\} \subset X$. Then the space $X = X/A$ is geodesic: indeed, given any $(a_n),(b_n) \in X$ the path $\eta\colon [0,1] \to X$ defined by $\eta(t) = ((1-t)a_n+tb_n)_{n=1}^\infty$ is a geodesic. We claim that $\Di{X,A}$ is not geodesic.

For any finite subset $F \subset \Z_{\geq 1}$, let $\alpha_F = (a^F_n)_{n=1}^\infty \in X$ be such that $a_n^F = 1+\frac{1}{n}$ if $n \in F$ and $a_n^F = 0$ otherwise. Let $\mathcal{F}_+$ (respectively $\mathcal{F}_-$) be the collection of finite subsets of $\Z_{\geq 1}$ of even (respectively odd) cardinality, and consider the diagrams $\sigma_\pm = \{\!\{ \alpha_F : F \in \mathcal{F}_\pm \}\!\} \in \Di{X,A}$. We claim that there are no geodesics in $\Di{X,A}$ from $\sigma_+$ to $\sigma_-$.

Note first that $d_\infty(\sigma_+,\sigma_-) = 1$. Indeed, we have $d(\alpha_F,\alpha_{F'}) > 1$ and $d(\alpha_{F'},A) > 1$ for any $F \in \mathcal{F}_+$ and $F' \in \mathcal{F}_-$, implying that $d_\infty(\sigma_+,\sigma_-) \geq 1$. On the other hand, for any $N \geq 1$ and any $B \subseteq [N]$ the sets $\{ F \in \mathcal{F}_+ : F \cap [N] = B \}$ and $\{ F \in \mathcal{F}_- : F \cap [N] = B \}$ are both countably infinite (where we write $[N]$ for $\{1,\ldots,N\}$), implying that for any $N$ there exists a bijection $\delta_N \colon \mathcal{F}_+ \to \mathcal{F}_-$ such that $F \cap [N] = \delta_N(F) \cap [N]$ for any $F \in \mathcal{F}_+$. We can then define a bijection $\gamma_N \colon \sigma_+ \to \sigma_-$ by setting $\gamma_N(\alpha_F) = \alpha_{\delta_N(F)}$, and it's easy to check that $d(\alpha,\gamma_N(\alpha)) < 1+\frac{1}{N}$ for any $\alpha \in \sigma_+$. This implies that $d_\infty(\sigma_+,\sigma_-) = 1$, as claimed.

Now suppose for contradiction that $\xi\colon [0,1] \to \Di{X,A}$ is a constant speed geodesic from $\sigma_+$ to $\sigma_-$. Then $d_\infty(\sigma_-,\xi(\frac{1}{2})) = \frac{1}{2} < \frac{3}{2}$, implying that there exists a point $\beta = (b_n)_{n=1}^\infty \in \xi(\frac{1}{2})$ (which we fix from now on) such that $|2-b_1| \leq d(\alpha_{\{1\}},\beta) < \frac{3}{2}$ and therefore $b_1 > \frac{1}{2}$. Let $F_0 = \{ n \in \Z_{\geq 1} : b_n > \frac{1}{2} \}$; as $\beta \in X = c_0$, we know that $|F_0| < \infty$ and that there exists $\varepsilon \in (0,\frac{1}{2})$ such that $b_n \geq \frac{1}{2}+\varepsilon$ for all $n \in F_0$ and $b_n \leq \frac{1}{2}-\varepsilon$ for all $n > \frac{1}{\varepsilon}$. Since $d(\beta,A) \geq \frac{1}{2}+\varepsilon > \frac{1}{2}$ and since $d_\infty(\sigma_\pm,\xi(\frac{1}{2})) = \frac{1}{2}$, there exist subsets $F_\pm \in \mathcal{F}_\pm$ such that $d(\alpha_{F_\pm},\beta) < \frac{1}{2}+\varepsilon$. Now if $n \in F_0$ then $|0-b_n| \geq \frac{1}{2}+\varepsilon$ and therefore we must have $a_n^{F_\pm} = 1+\frac{1}{n}$; on the other hand, if $n \notin F_0$ then $b_n \leq \frac{1}{2}$ and either $\frac{1}{n} \geq \varepsilon$ or $b_n \leq \frac{1}{2}-\varepsilon$, so in either case we have $|1+\frac{1}{n}-b_n| \geq \frac{1}{2}+\varepsilon$ and therefore we must have $a_n^{F_\pm} = 0$. This implies that $F_- = F_0 = F_+$, which is impossible since $\mathcal{F}_- \cap \mathcal{F}_+ = \varnothing$; thus there are no geodesics in $\Di{X,A}$ from $\sigma_+$ to $\sigma_-$, as claimed.
\end{example}

Note that, for $\sigma_\pm \in \Di{X,A}$ as in Example~\ref{ex:non-geodesic}, there are no bijections $\gamma\colon \sigma_+ \to \sigma_-$ such that $\sup_{x \in \sigma_+} d(x,\gamma(x)) = d_\infty(\sigma_+,\sigma_-)$. Indeed, if there was such a bijection then the map $[0,1] \to \Di{X,A}$ sending $t$ to the diagram $\mset{(1-t)x+t\gamma(x) : x \in \sigma_+}$ would be a geodesic from $\sigma_+$ to $\sigma_-$, but we have shown that no such geodesics exist.

\bibliographystyle{amsplain}
\bibliography{BOTTLENECK_CURRENT}

\end{document}